\documentclass[10pt, reqno, backref]{amsart}

\usepackage[all]{xy}
\usepackage{url}
\usepackage{amssymb}
\usepackage{hyperref}
\usepackage[margin=.95 in]{geometry}
\usepackage{color,soul}
\usepackage{amsthm}
\numberwithin{equation}{subsection}
\usepackage{mathtools}

\DeclareMathOperator{\Spec}{Spec}

\DeclareMathOperator{\Mspec}{MaxSpec}
\DeclareMathOperator{\id}{id}
\DeclareMathOperator{\Frac}{Frac}
\DeclareMathOperator{\colim}{colim}
\DeclareMathOperator{\Gal}{Gal}

\theoremstyle{plain}

\theoremstyle{definition}

\theoremstyle{plain}
\newtheorem{Theorem}[subsection]{Theorem}
\newtheorem{Proposition}[subsection]{Proposition}
\newtheorem{Lemma}[subsection]{Lemma}
\newtheorem{Corollary}[subsection]{Corollary}

\newtheorem{Conjecture}[subsection]{Conjecture}

\theoremstyle{definition}
\newtheorem{Definition}[subsection]{Definition}
\newtheorem{Example}[subsection]{Example}

\newtheorem{Notation}[subsection]{Notation}

\newtheorem{Remark}[subsection]{Remark}

\theoremstyle{plain}
\newtheorem{innercustomthm}{Theorem}
\newenvironment{customthm}[1]
  {\renewcommand\theinnercustomthm{#1}\innercustomthm}
  {\endinnercustomthm}

\theoremstyle{plain}

\normalfont
\usepackage[T1]{fontenc}{}
\usepackage{setspace}

\usepackage{tikz-cd}

\makeatletter
\def\@tocline#1#2#3#4#5#6#7{\relax
  \ifnum #1>\c@tocdepth 
  \else
    \par \addpenalty\@secpenalty\addvspace{#2}%
    \begingroup \hyphenpenalty\@M
    \@ifempty{#4}{%
      \@tempdima\csname r@tocindent\number#1\endcsname\relax
    }{%
      \@tempdima#4\relax
    }%
    \parindent\z@ \leftskip#3\relax \advance\leftskip\@tempdima\relax
    \rightskip\@pnumwidth plus4em \parfillskip-\@pnumwidth
    #5\leavevmode\hskip-\@tempdima
      \ifcase #1
       \or\or \hskip 1em \or \hskip 2em \else \hskip 3em \fi%
      #6\nobreak\relax
    \dotfill\hbox to\@pnumwidth{\@tocpagenum{#7}}\par
    \nobreak
    \endgroup
  \fi}
\makeatother
\usepackage{hyperref}

\title{Essential finite generation of extensions of valuation rings}
\author{Rankeya Datta\vspace{-2ex}}
\address{}
\email{}
\thanks{}

\begin{document}

\maketitle

\begin{abstract}
Given a generically finite local extension of valuation rings 
$V \subset W$, the question of whether $W$ is the localization
of a finitely generated $V$-algebra is significant
for approaches to the
problem of local uniformization of valuations using 
ramification theory. Hagen Knaf proposed a characterization of
when $W$ is essentially of finite type over $V$ in terms of classical
invariants of the extension of associated valuations. Knaf's
conjecture has been  
verified in important special cases by
Cutkosky and Novacoski using local uniformization of Abhyankar valuations 
and resolution of singularities of excellent
surfaces in arbitrary characteristic, and by Cutkosky
for valuation rings of function fields of 
characteristic $0$ using embedded resolution of singularities. 
In this paper we prove
 Knaf's conjecture in full generality. 
\end{abstract}

\section{Introduction}

Let $L/K$ be a finite field extension. Given a domain $R$ with 
fraction field $K$, results that characterize when the integral closure of $R$ in $L$ is a finite type 
$R$-algebra have fundamental applications in
algebraic geometry, commutative algebra and number theory. For instance, finite generation of integral closures
was studied extensively, among others, by Krull, Akizuki, Noether, Zariski, Grothendieck and especially Nagata, resulting in applications to fundamental topics such as rings of algebraic integers and resolution of singularities. We investigate a local
valuative analogue of the finite generation of 
integral closures in this paper. 
Let $\omega$ be a valuation of $L$ with valuation ring 
$(\mathcal{O}_\omega, \mathfrak{m}_\omega, \kappa_\omega)$ and value
group $\Gamma_\omega$. Let $\nu$ 
be the restriction of $\omega$ to $K$ with valuation ring
$(\mathcal{O}_\nu, \mathfrak{m}_\nu, \kappa_\nu)$ and value group
$\Gamma_\nu$. Inclusion induces a local homomorphism
\[
(\mathcal{O}_\nu, \mathfrak{m}_\nu, \kappa_\nu) \hookrightarrow
(\mathcal{O}_\omega, \mathfrak{m}_\omega, \kappa_\omega),
\]
and the valuation ring $\mathcal{O}_\omega$ is a local ring of
the integral closure of $\mathcal{O}_\nu$ in $L$ \cite[Chap.\ VI, $\mathsection 8.6$,
Prop.\ 6]{Bou98}. 
Thus, as a local version of
the question of the finite generation of integral closures, it is natural to ask when 
$\mathcal{O}_\omega$ is the localization of a finite type 
$\mathcal{O}_\nu$-algebra. Knaf 
proposed the following 
necessary and sufficient condition 
in terms of
classical invariants of the extension $\omega/\nu$
 \cite[Conjecture 1.2]{CN19}.

\begin{Conjecture}
\label{conj:Knaf}
Let $L/K$ be a finite field extension, $\omega$ be a valuation of
$L$ and $\nu$ be the restriction of $\omega$ to $K$. 
Let $L^h$ (resp. $K^h$) denote the fraction field of the 
Henselization of $\mathcal{O}_\omega$ (resp. $\mathcal{O}_\nu$).
Then $\mathcal{O}_\omega$ is essentially of finite
type over $\mathcal{O}_\nu$ if and only if both the following 
conditions are satisfied:
\begin{enumerate}
	\item $[L^h: K^h] = [\Gamma_\omega:\Gamma_\nu]
[\kappa_\omega:\kappa_\nu]$.
	\item $\epsilon(\omega|\nu) = 
[\Gamma_\omega:\Gamma_\nu]$, where $\epsilon(\omega|\nu)$ is the
cardinality of 
$\textrm{$\{x \in \Gamma_{\omega, \geq 0}: x < y$ for all 
$y \in \Gamma_{\nu, > 0}\}$.}$
\end{enumerate}
\end{Conjecture}

\noindent Here for a totally ordered abelian group $\Gamma$, we define
\[
\textrm{$\Gamma_{\geq 0} \coloneqq \{x \in \Gamma: x \geq 0\}$ and $\Gamma_{> 0}
\coloneqq \{x \in \Gamma: x > 0\}$}. 
\]
In ramification theory, $[\Gamma_\omega:\Gamma_\nu]$ is called the
\emph{ramification index}, 
$[\kappa_\omega:\kappa_\nu]$ the \emph{inertia index} and $\epsilon(\omega|\nu)$ 
the \emph{initial index} of the extension $\omega/\nu$.
The first equality $[L^h: K^h] = [\Gamma_\omega:\Gamma_\nu][\kappa_\omega:\kappa_\nu]$ is the 
assertion that the extension $\omega/\nu$ is \emph{defectless}; see 
Definition \ref{def:defect} for the notion of defect. The second equality
$\epsilon(\omega|\nu) = [\Gamma_\omega:\Gamma_\nu]$ means that
every element of the quotient $\Gamma_\omega/\Gamma_\nu$ is the class of 
some element of $\Gamma_{\omega, \geq 0}$. Indeed, if 
$x_1, x_2$ are distinct elements of 
$\mathcal{S} \coloneqq \{x \in \Gamma_{\omega, \geq 0}: x < y$ for all 
$y \in \Gamma_{\nu, > 0}\}$, then we may assume without loss
of generality that $0 \leq x_1 < x_2$. Consequently, $0 < x_2 - x_1 \leq x_2$, and so, $x_2 - x_1$ cannot be an element of $\Gamma_\nu$ because $x_2$ is strictly smaller than every element
of $\Gamma_{\nu, > 0}$. This shows that every 
element of $\mathcal S$ is the representative of a distinct
class of $\Gamma_\omega/\Gamma_\nu$. Thus, $\epsilon(\omega|\nu) = [\Gamma_\omega:\Gamma_\nu]$ would mean that the classes of the
elements of $\mathcal S$ constitute the whole group $\Gamma_\omega/\Gamma_\nu$.


The question of the essential finite generation of 
extension of valuation rings arises 
naturally in approaches to the open problem of
 local uniformization of valuations using ramification theory. For
example, an affirmative answer to this question for extensions of Abhyankar valuations 
is an important ingredient in Knaf and
Kuhlmann's proof of the local uniformization of Abhyankar 
valuations \cite{KK05}. In addition, the essential finite generation
of extensions of valuation rings also features in Knaf and Kuhlmann's 
valuation-theoretic argument
of the local uniformization of a valuation in a finite extension
of its fraction field \cite{KK09}. Conjecture \ref{conj:Knaf}
can be viewed as a generalization of the beautiful ramification-theoretic characterization of 
the module-finiteness of the integral closure
of a valuation ring in a finite extension of its fraction field 
\cite[Chap.\ VI, $\mathsection 8.5$, Thm.\ 2]{Bou98}.
We refer the interested reader to \cite{CN19, Cut19} for 
additional background on this problem. 

Conjecture \ref{conj:Knaf} is known in specific cases, often using different
techniques. The necessity of conditions (1) and (2)
for the essential finite generation of $\mathcal{O}_\omega$ over $\mathcal{O}_\nu$
was proved by Knaf; his argument is reproduced in \cite[Thm.\ 4.1]{CN19} 
(see also Remark \ref{rem:ZMT} for a different
approach using Zariski's Main Theorem). 
The sufficiency of conditions (1) and (2) for the 
essential finite generation of $\mathcal{O}_\omega$ over $\mathcal{O}_\nu$ 
is known when 
\begin{itemize}
 	\item $L/K$ is normal using the transitive action of 
	$\Gal(L/K)$ on the fibers of the integral closure of 
	$\mathcal{O}_\nu$ in $L$ \cite[Cor.\ 2.2]{CN19};

	\item $\kappa_\omega/\kappa_\nu$ is separable using the
	theory of Henselian elements \cite[Thm.\ 1.3]{KN14}; 

	\item $\omega$ is the unique extension of $\nu$ to $L$ using
	the theory of defect \cite[Cor.\ 2.2]{CN19};

	\item $\nu$ is centered on an excellent local two dimensional 
	domain with fraction field $K$ using resolution of 
	singularities for excellent surfaces \cite[Thm.\ 1.4]{CN19};
 
	\item $\nu$ is an Abhyankar valuation of a function field $K/k$	by \cite[Thm.\ 1.5]{CN19} and \cite[Thm.\ 1.7]{Cut20}
	using local uniformization of Abhyankar valuations;
   
	\item $K$ is the function field over a field of
characteristic $0$ using an explicit form of embedded resolution
of singularities \cite[Thm.\ 1.3]{Cut19}.
\end{itemize}

We will give a uniform argument that settles
Conjecture \ref{conj:Knaf} in full generality. Recall that the implication that remains to be 
shown is that if conditions (1) and (2) of 
Conjecture \ref{conj:Knaf} hold, then $\mathcal{O}_\omega$ is
essentially of finite type over $\mathcal{O}_\nu$. We take as our 
starting point the veracity of Conjecture \ref{conj:Knaf} for unique extensions 
of valuations \cite[Cor.\ 2.2]{CN19}. 
Note that
if $\omega^h$ (resp. $\nu^h$) is the valuation of $L^h$ (resp. $K^h$) whose valuation ring
is the Henselization $\mathcal{O}^h_\omega$ (resp. $\mathcal{O}^h_\nu$), then $\omega^h$
is the unique extension of $\nu^h$ to $L^h$ 
up to equivalence of valuations by the Henselian property (see Corollary \ref{cor:Henselian-valuations}). 
Moreover,
Henselizations do not alter value groups 
\cite[\href{https://stacks.math.columbia.edu/tag/0ASK}{Tag 0ASK}]{stacks-project} 
and residue fields, which means that the
ramification, residue and inertia indexes of $\omega^h/\nu^h$
coincide with those of $\omega/\nu$. 
Thus, assuming (1) and (2) in Conjecture \ref{conj:Knaf}, it
follows that $\mathcal{O}_\omega^h$ is essentially of finite type
over $\mathcal{O}_\nu^h$. Using this observation, we prove
Conjecture \ref{conj:Knaf} by establishing the following result.

\begin{Theorem}
\label{thm:main-theorem}
Let $L/K$ be a finite field extension, $\omega$ be a valuation of 
$L$ with valuation ring $(\mathcal{O}_\omega, \mathfrak{m}_\omega, \kappa_\omega)$ 
and value group $\Gamma_\omega$, and $\nu$ be the restriction of $\omega$ to $K$
with valuation ring $(\mathcal{O}_\nu, \mathfrak{m}_\nu, \kappa_\nu)$
and value group $\Gamma_\nu$. 
Let $\mathcal{O}^h_\omega$ (resp. $\mathcal{O}^h_\nu$) be the Henselization of $\mathcal{O}_\omega$ (resp. $\mathcal{O}_\nu$)
with fraction field $L^h$ (resp. $K^h$). 
Then the following are equivalent.
\begin{enumerate}
	\item[(i)] $\mathcal{O}_\omega$ is essentially of finite
	presentation over $\mathcal{O}_\nu$.
	\item[(ii)] $\mathcal{O}_\omega$ is essentially of finite type
	over $\mathcal{O}_\nu$.
	\item[(iii)] $[L^h:K^h] = [\Gamma_\omega:\Gamma_\nu][\kappa_\omega:\kappa_\nu]$ and $\epsilon(\omega|\nu) = [\Gamma_\omega:\Gamma_\nu]$.
	\item[(iv)] $\mathcal{O}^h_\omega$ is essentially of finite
	type over $\mathcal{O}^h_\nu$.
	\item[(v)] $\mathcal{O}^h_\omega$ is a module finite 
	$\mathcal{O}^h_\nu$-algebra.
\end{enumerate}
\end{Theorem}

Instead of using known cases of local uniformization or 
resolution of singularities, we use Knaf's result \cite[Thm.\ 4.1]{CN19}
(see also Remark \ref{rem:ZMT}) to
reduce the proof of Theorem \ref{thm:main-theorem} to showing
that (v) $\Rightarrow$ (i). We then prove this implication 
by analyzing the
behavior of integral maps under base change along Henselizations. 
This consideration turns out to be independent of
valuation theory and is carried out in Section \ref{sec:Henselization and base change}. 
This section is
the heart of our paper and does most of the heavy lifting for the
proof of (v) $\Rightarrow$ (i). The key is the 
approximation result of Corollary \ref{cor:colimit}, which, 
combined the fact that valuation rings are 
maximal subrings of a field with respect to the partial order
induced by domination of local rings, leads to a proof of 
Theorem \ref{thm:main-theorem}.

The paper is structured as follows. In Section \ref{sec:Conventions and basic terminology} 
we fix our conventions
for the paper. Section \ref{sec:Henselization and base change}
examines base change properties along Henselizations. In
Section \ref{sec:valuation-background} we collect some basic 
results and definitions about extensions of valuation rings, Henselian valuation
rings and the ramification theory of extensions of
valuations. Finally, we prove 
Theorem \ref{thm:main-theorem} in Section \ref{sec:proof}.

\noindent \textbf{Acknowledgments:} The author thanks Dale 
Cutkosky for helpful conversations and the referee for their 
helpful suggestions that improved the clarity of the paper.

\section{Conventions and basic terminology}
\label{sec:Conventions and basic terminology}
All rings are commutative with a multiplicative identity. For a ring $A$, $\Mspec(A)$ will denote the set of
maximal ideals of $A$. Note that rings in this paper will rarely
be noetherian.

The term \emph{local ring} will mean a ring $(R, \mathfrak m, \kappa)$ with a unique maximal ideal
$\mathfrak{m}$ and residue field $\kappa$. Please note that local rings
are not necessarily noetherian. We will use $R^h$ to denote the Henselization of the local ring $R$ with respect to 
the maximal ideal $\mathfrak{m}$. Recall that $R^h$ is a faithfully
flat local extension of $R$ whose maximal ideal is the expansion
of the maximal ideal of $R$ and whose residue field is isomorphic
to the residue field of $R$ 
\cite[\href{https://stacks.math.columbia.edu/tag/07QM}{Tag 07QM}]{stacks-project}. 
We will say that a local
ring $(B, \mathfrak{m}_B)$ \emph{dominates} a local ring $(A, \mathfrak{m}_A)$ if $A \subset B$ and $\mathfrak{m}_A = \mathfrak{m}_B \cap A$. Recall that a valuation ring of a field 
$K$ is a local subring of $K$
that is maximal in the collection of local subrings of $K$ 
under the partial order induced by domination of local rings.
Whenever we talk about an extension of valuation rings $V \subset W$, 
we will
always assume $W$ dominates $V$.

Valuation rings also arise from valuations, which are denoted additively in this paper. We assume the reader is familiar with
valuations and valuation rings, and we will skip their 
definitions. A good introduction to valuation theory is 
\cite[Chap.\ VI]{Bou98}.

Let $B$ be an $A$-algebra. We say $B$ is 
\emph{essentially of finite type over $A$} if $B$ is the 
localization of a finite type $A$-algebra. We say $B$ is
\emph{essentially of finite presentation over $A$} if $B$ is the
localization of a finitely presented $A$-algebra.

While a finitely presented algebra is always of finite type, the converse is not true in a non-noetherian setting. For example, if $k$
is a field and $A = k[x_n \colon n \in \mathbf{N}]$ is the
polynomial ring in infinitely many variables over $k$, then
the canonical map $A \twoheadrightarrow k$ obtained by killing
all the variables is of finite type but
not of finite presentation because the ideal $(x_n \colon n \in \mathbf{N})$ is not finitely generated 
\cite[\href{https://stacks.math.columbia.edu/tag/00R2}{Tag 00R2}]{stacks-project}. However, finite generation and finite presentation 
often coincide in the valuative setting because of the 
following result and the fact that any torsion-free module over a
valuation ring is free \cite[Chap.\ VI, $\mathsection 3.6$, Lem.\ 1]{Bou98}.

\begin{Lemma}\cite[Cor.\ (3.4.7)]{RG71}
\label{lem:finite-pres}
Let $A$ be a domain and $B$ be a finite type $A$-algebra. 
If $B$ is $A$-flat, then $B$ is a finitely presented $A$-algebra.
\end{Lemma}

\section{Henselization and base change}
\label{sec:Henselization and base change}

In this section we establish some base change properties of integral
maps along Henselizations. The results do not use any valuation
theory. The non-valuative considerations of this
section will provide the main ingredients for the proof of Theorem 
\ref{thm:main-theorem}. We will frequently use the fact that the 
Henselization of a local ring is flat
\cite[\href{https://stacks.math.columbia.edu/tag/07QM}{Tag 07QM}]{stacks-project}, that flat maps satisfy 
the Going-Down property 
\cite[\href{https://stacks.math.columbia.edu/tag/00HS}{Tag 00HS}]{stacks-project} and that the property of being an 
integral ring map is preserved under base change 
\cite[\href{https://stacks.math.columbia.edu/tag/02JK}{Tag 02JK}]{stacks-project}.

We first recall a characterization of Henselian
local domains that  will be important for the results that follow.

\begin{Lemma}
\label{lem:Nagata-hensel-domains}
Let $(R, \mathfrak m)$ be a local domain. The following are equivalent:
\begin{enumerate}
	\item $R$ is Henselian.
	\item For every integral extension $R \hookrightarrow A$, if 	$A$ is a domain then
	$A$ is a local ring.
\end{enumerate}
If the equivalent conditions hold, then any integral extension of
$R$ that is also a domain is Henselian.
\end{Lemma}

\begin{proof}[Indication of proof]
The equivalence follows from 
\cite[Chap.\ VII, Thm.\ (43.12)]{Nag75}. The fact that integral 
extension domains of $R$ 
are Henselian follows by \cite[Chap.\ VII, Cor.\ (43.13)]{Nag75}. 
Note that in Nagata's terminology, an integral extension of a domain 
is automatically a domain \cite[Chap.\ I, Pg.\ 30]{Nag75}.
\end{proof}

The next result highlights a key base change property along 
Henselizations.

\begin{Lemma}
\label{lem:Henselization-base-change}
Let $(R, \mathfrak m)$ be a local ring, $\varphi: R \rightarrow A$ be a ring map and
$i: R \rightarrow R^h$ be the canonical map from $R$ to its Henselization $R^h$.
Suppose $\mathfrak P \in \Spec(A)$ contracts to $\mathfrak m \in \Spec(R)$, that is,
$\varphi^{-1}(\mathfrak{P})= \mathfrak{m}$. Consider the induced map $\id_A \otimes i: A \rightarrow A \otimes_R R^h$.
\begin{enumerate}
	\item The fiber of $\Spec(\id_A \otimes i): \Spec(A \otimes_R R^h) \rightarrow \Spec(A)$
	 over $\mathfrak P$ is a singleton.
	\item If $\mathfrak Q$ is the unique prime ideal of $A \otimes_R R^h$ that contracts to $\mathfrak 		P$,
	then $((A \otimes_R R^h)_{\mathfrak Q})^h \cong (A_{\mathfrak P})^h$.
\end{enumerate}
\end{Lemma}

\begin{proof}
Consider the commutative diagram
\begin{equation}
\label{diag:base-change}
\begin{tikzcd}
A \arrow[r, "\id_A \otimes i"]
&A \otimes_R R^h\\
R \arrow[r, "i"] \arrow[swap]{u}{\varphi}
& R^h \arrow[swap]{u}{\varphi \otimes \id_{R^h}}.
\end{tikzcd}
\end{equation}
Since $i$ is faithfully flat, so is $\id_A \otimes i$.
Therefore,
$\Spec(\id_A \otimes i): \Spec(A \otimes_R R^h) \rightarrow \Spec(A)$
is surjective.

\noindent (1) Let
$\kappa(\mathfrak P)$ (resp. $\kappa(\mathfrak m)$) denote the residue field of 
$\mathfrak P$ (resp. $\mathfrak m$). Then the fiber of $\Spec(\id_A \otimes i)$ over $\mathfrak P$
can be identified with $\Spec(\kappa(\mathfrak P) \otimes_A (A \otimes_R R^h))$.
Now,
\[
\kappa(\mathfrak P) \otimes_A (A \otimes_R R^h) \cong \kappa(\mathfrak P) \otimes_R R^h
\cong \kappa(\mathfrak P) \otimes_{\kappa(\mathfrak m)} (\kappa(\mathfrak m) \otimes_R R^h).
\]
The maximal ideal $\mathfrak{m}^h$ of $R^h$ is $\mathfrak{m}R^h$  and 
the induced map of residue fields $\kappa(\mathfrak m) \rightarrow \kappa(\mathfrak{m}^h)$ 
is an isomorphism \cite[\href{https://stacks.math.columbia.edu/tag/07QM}{Tag 07QM}]{stacks-project}. 
Therefore $\kappa(\mathfrak m) \otimes_R R^h = \kappa(\mathfrak{m}^h)$ and 
\[
\kappa(\mathfrak P) \otimes_{\kappa(\mathfrak m)} (\kappa(\mathfrak m) \otimes_R R^h) \cong \kappa(\mathfrak P),
\]
that is, $\Spec(\kappa(\mathfrak P) \otimes_A (A \otimes_R R^h))$ is the spectrum of a field.
This proves (1).

\noindent (2) There exists a unique prime ideal $\mathfrak Q$ of $A \otimes_R R^h$ that contracts
to $\mathfrak P$ by (1). By the commutativity of (\ref{diag:base-change}), 
$\mathfrak Q$ contracts to $\mathfrak{m}^h$ along 
$\varphi \otimes \id_{R^h}: R^h \rightarrow A \otimes_R R^h$ because $\mathfrak{m}^h$ is
the unique prime ideal of $R^h$ that contracts $\mathfrak m$ in $R$. 
The rest of (2) now follows from 
\cite[\href{https://stacks.math.columbia.edu/tag/08HU}{Tag 08HU}]{stacks-project}. 
\end{proof}

We now focus on the base change properties of integral
ring maps along Henselizations. When we use the term 
`integral ring map', we do not necessarily mean an integral extension. For instance, a surjective ring map is 
an integral map but
not an integral extension.

\begin{Lemma}
\label{lem:Henselization-base-change-integral}
Let $(R, \mathfrak m)$ be a local ring, $\varphi: R \rightarrow A$ be an integral
ring map and $i: R \rightarrow R^h$
be the canonical map from $(R, \mathfrak m)$ to its Henselization $(R^h, \mathfrak{m}^h)$. 
Then the map 
\[
\Spec(\id_A \otimes i): \Spec(A \otimes_R R^h) \rightarrow \Spec(A)
\]
induced by $\id_A \otimes i: A \rightarrow A \otimes_R R^h$ has the following properties:
\begin{enumerate}
	\item For every maximal ideal $\mathfrak M$ of $A$, the fiber of $\Spec(\id_A \otimes i)$
	over $\mathfrak M$ is a singleton.
	\item Let $\mathfrak Q \in \Spec(A \otimes_R R^h)$. The following are equivalent:
	\begin{enumerate}
	\item[(2a)] $\mathfrak Q$ is a maximal ideal of $A \otimes_R R^h$.
	\item[(2b)] $\mathfrak Q$ contracts to a maximal ideal
	of $\Spec(A)$.
	\item[(2c)] $\mathfrak Q$ contracts to $\mathfrak{m}^h$.
	\end{enumerate}
	\item $\Spec(\id_A \otimes i)$ induces a bijection 
	$\Mspec(A \otimes_R R^h) \longleftrightarrow \Mspec(A)$.
\end{enumerate}
\end{Lemma}

\begin{proof}
Consider the commutative diagram
\[
\begin{tikzcd}
A \arrow[r, "\id_A \otimes i"]
&A \otimes_R R^h\\
R \arrow[r, "i"] \arrow[swap]{u}{\varphi}
& R^h \arrow[swap]{u}{\varphi \otimes \id_{R^h}}.
\end{tikzcd}
\]
Since $\varphi$ is integral, so is $\varphi \otimes \id_{R^h}$.
Moreover, 
$\Spec(\id_A \otimes i): \Spec(A \otimes_R R^h) \rightarrow \Spec(A)$
is surjective because $\id_A \otimes i$ is faithfully flat by base change.

\noindent (1) $\varphi^{-1}(\mathfrak M) = \mathfrak m$ since $\varphi$ is integral and $\mathfrak M$
is maximal. Then (1) follows 
by part (1) of Lemma \ref{lem:Henselization-base-change}.

\noindent (2) Suppose $\mathfrak Q$ is a maximal ideal of $A \otimes_R R^h$. Then
$\mathfrak Q$ contracts to $\mathfrak{m}^h$ in $R^h$ along $\varphi \otimes \id_{R^h}$ because this map
is integral. Thus, (2a) $\Rightarrow$ (2c). 

Suppose $\mathfrak Q$ contracts to $\mathfrak{m}^h$ in $R^h$, and hence to $\mathfrak{m}$ in $R$. 
Then the
contraction $\mathfrak{Q}^c$ of $\mathfrak Q$ to $A$ must be maximal because $\mathfrak{Q}^c$ contracts to $\mathfrak m$
along the integral map $\varphi$, and only maximal ideals can contract to maximal ideals
along integral maps. This proves (2c) $\Rightarrow$ (2b).

Finally, suppose $\mathfrak Q$ contracts to a maximal ideal $\mathfrak M$ of $A$. Since 
$\varphi^{-1}(\mathfrak M) = \mathfrak m,$ 
it follows by the commutativity of the above diagram that $\mathfrak Q$ 
contracts to $\mathfrak{m}^h$ along the integral ring map
$\varphi \otimes \id_{R^h}$. 
Then $\mathfrak Q$ must be maximal, thereby establishing (2b) $\Rightarrow$ (2a).

\noindent (3) The equivalent statements of part (2) tell us that the inverse image of $\Mspec(A)$ under 
$\Spec(\id_A \otimes i)$ is precisely $\Mspec(A \otimes_R R^h)$, and part (1)
shows that the induced map $\Mspec(A \otimes_R R^h) \rightarrow \Mspec(A)$
is both injective and surjective.
\end{proof}

An analogue of Lemma \ref{lem:Henselization-base-change-integral} exists for minimal primes.

\begin{Lemma}
\label{lem:minimal-primes-base-change-Henselization}
Let $(R,\mathfrak m)$ be an integrally closed domain, and $\varphi: R \hookrightarrow A$ be an integral extension of domains. Then we 
have the following:
\begin{enumerate}
	\item $R^h$ is an integrally closed domain
	and $\Frac(R) \otimes_R R^h = \Frac(R^h)$.
	\item Let $\mathfrak Q \in \Spec(A \otimes_R R^h)$.
	The following are equivalent:
	\begin{enumerate}
	\item[(2a)] $\mathfrak Q$ is a minimal prime of $A \otimes_R R^h$. 
	\item[(2b)] $\mathfrak Q$ lies over $(0)$ in $A$.
	\item[(2c)] $\mathfrak Q$ lies over $(0)$ in $R^h$.
	\end{enumerate} 
	\item There is a bijection $\{$minimal prime of $A \otimes_R R^h \} \longleftrightarrow
	 \Spec(\Frac(A) \otimes_{\Frac(R)} \Frac(R^h))$.
	\item If $\Frac(A)$ is a finite extension of $\Frac(R)$, then $A \otimes_R R^h$
	has finitely many minimal primes.
\end{enumerate}
\end{Lemma}

\begin{proof}
(1) That $R^h$ is an integrally closed domain is a well-known permanence property of Henselization; see 
\cite[\href{https://stacks.math.columbia.edu/tag/06DI}{Tag 06DI}]{stacks-project}.
Since $R^h$ is a colimit of local \'etale extensions, 
$\Frac(R^h)$ is an algebraic extension of $\Frac(R)$.
So $\Frac(R) \otimes_R R^h$ is a field because it contains $\Frac(R)$ and is contained
in $\Frac(R^h)$.  
Since $\Frac(R) \otimes_R R^h$ is a localization of $R^h$, we get $\Frac(R) \otimes_R R^h = \Frac(R^h)$.

\noindent (2) Let $i: R \rightarrow R^h$ be the canonical map. 
Consider the commutative diagram
\[
\begin{tikzcd}
A \arrow[r, "\id_A \otimes i"]
&A \otimes_R R^h\\
R \arrow[r, "i"] \arrow[swap]{u}{\varphi}
& R^h \arrow[swap]{u}{\varphi \otimes \id_{R^h}}.
\end{tikzcd}
\]
Since $\id_A \otimes i: A \rightarrow A \otimes_R R^h$ is flat, 
a minimal prime of $A \otimes_R R^h$ must contract to the unique minimal prime $(0)$ of $A$
by Going-Down. 
This proves (2a) $\Rightarrow$ (2b).

Suppose $\mathfrak Q$ contracts to $(0)$ in $A$. Then $\mathfrak Q$ must contract to 
$(0)$ in $R$ because
$\varphi$ is injective.
Let
\[
\mathfrak p \coloneqq (\varphi \otimes \id_{R^h})^{-1}(\mathfrak Q).
\]
By the commutativity of the above diagram, $\mathfrak p$ contracts to $(0)$ in $R$. 
But (1) shows the generic fiber of $i$ is a singleton, consisting
of the unique minimal prime $(0)$ of $R^h$. 
Consequently, $\mathfrak p = (0)$, proving (2b) $\Rightarrow$ (2c).

Assume (2c). If $\mathfrak Q$ is not a minimal prime of $A \otimes_R R^h$, then we can find
 $\mathfrak{Q}' \in \Spec(A \otimes_R R^h)$ such that
\[
\mathfrak{Q}' \subsetneq \mathfrak{Q}.
\]
Then $(\varphi \otimes \id_{R^h})^{-1}(\mathfrak{Q}') = (0)$, which is a contradiction
because $\varphi \otimes \id_{R^h}$ is an integral extension, and integral extensions have
zero dimensional fibers 
\cite[\href{https://stacks.math.columbia.edu/tag/00GT}{Tag 00GT}]{stacks-project}. 
Thus, (2c) $\Rightarrow$ (2a).

\noindent (3) By part (2), the set of minimal primes of $A \otimes_R R^h$ is precisely the
generic fiber of $\id_A \otimes i: A \rightarrow A \otimes_R R^h$, which is in bijection with 
$\Spec(\Frac(A) \otimes_A (A \otimes_R R^h))$. The assertion now follows because
\[
\Frac(A) \otimes_A (A \otimes_R R^h) \cong \Frac(A) \otimes_{\Frac(R)} (\Frac(R) \otimes_R R^h)
\cong \Frac(A) \otimes_{\Frac(R)} \Frac(R^h),
\]
where the last isomorphism is a consequence of part (1).

\noindent (4) If $\Frac(A)$ is a finite extension of $\Frac(R)$, then 
$\Frac(A) \otimes_{\Frac(R)} \Frac(R^h)$ is a finite 
$\Frac(R^h)$-algebra.
Consequently, 
$\Spec(\Frac(A) \otimes_{\Frac(R)} \Frac(R^h))$ is a finite set. We are then done 
by part (3).
\end{proof}

The next result is well-known. We include a proof for the 
reader's convenience.

\begin{Lemma}
\label{lem:direct-product-domains}
Let $A$ be a ring such that for all maximal ideals $\mathfrak m$ of $A$, $A_{\mathfrak m}$ is a domain.
\begin{enumerate}
	\item If $\mathfrak p$ and $\mathfrak q$ are distinct minimal primes of $A$, then 
	$\mathfrak p + \mathfrak q = A$. 
	\item If $A$ has finitely many distinct minimal primes $\mathfrak{p}_1, \dots, \mathfrak{p}_n$, then the
	canonical map $A \rightarrow A/\mathfrak{p}_1 \times \dots \times A/\mathfrak{p}_n$ is an
	isomorphism.
\end{enumerate}
\end{Lemma}

\begin{proof}
(1) Suppose for contradiction that $\mathfrak p + \mathfrak q \subsetneq A$. 
Then there exists a maximal ideal $\mathfrak m$
of $A$ such that $\mathfrak p + \mathfrak q \subseteq \mathfrak m$. Now both $\mathfrak{p}A_{\mathfrak m}$ and
$\mathfrak{q}A_{\mathfrak m}$ are distinct minimal prime ideals of the domain $A_{\mathfrak m}$, which is impossible.

\noindent (2) 
The hypothesis implies that $A$ is reduced, that is,
$\mathfrak{p}_1 \cap \dots \cap \mathfrak{p}_n = (0).$
Since $\mathfrak{p}_i + \mathfrak{p}_j = A$ for $i \neq j$, the result now follows by the Chinese
Remainder Theorem
\cite[\href{https://stacks.math.columbia.edu/tag/00DT}{Tag 00DT}]{stacks-project}.
\end{proof}

\begin{Proposition}
\label{prop:decomposition-henselization-base-change}
Let $(R, \mathfrak m)$ be a local domain that is integrally closed in its 
fraction field $K$. Let $L$ be a finite field extension of $K$ and
let $A$ be the integral closure of $R$ in $L$. 
We have the following:
\begin{enumerate}
	\item $A$ has finitely many maximal ideals, that is, $A$ is semi-local.
	\item $A \otimes_R R^h$ is a semi-local ring.
	\item If $\mathfrak M$ is a maximal ideal of $A \otimes_R R^h$, then 
	$(A \otimes_R R^h)_{\mathfrak M}$ is an integrally closed domain.	
	\item $A \otimes_R R^h$ has finitely many minimal primes.
	\item 
	Each maximal ideal $\mathfrak M$ of $A \otimes_R R^h$ contains
	a unique minimal prime $\mathfrak p$, and conversely, 
	$\mathfrak M$ is the unique
	maximal ideal that contains $\mathfrak p$.
	Moreover, the canonical map 
	$A \otimes_R R^h \rightarrow (A\otimes_R R^h)_{\mathfrak M}$ has kernel
	$\mathfrak p$ and induces an isomorphism
	\[
	\frac{A \otimes_R R^h}{\mathfrak{p}} \cong (A \otimes_R R^h)_{\mathfrak M}.
	\]
	Consequently, $(A \otimes_R R^h)_{\mathfrak M}$ is Henselian.
	\item If $\mathfrak{m}_1, \dots, \mathfrak{m_n}$ are the maximal ideals of $A$ (assumed to be distinct), then
	\[
	A \otimes_R R^h \cong (A_{\mathfrak{m}_1})^h \times \dots \times (A_{\mathfrak{m}_n})^h.
	\]
	Moreover, if $\mathfrak{M}_i$ is the unique prime ideal of $A \otimes_R R^h$ that
	contracts to $\mathfrak{m}_i$, then $\mathfrak{M}_i$ is maximal and 
	\[
	(A \otimes_R R^h)_{\mathfrak{M}_i} \cong (A_{\mathfrak{m}_i})^h.
	\]
	\item The sets $\Mspec(A), \Mspec(A \otimes_R R^h), \Spec(L \otimes_K \Frac(R^h))$
	and $\{$minimal prime of $A \otimes_R R^h \}$ have the same cardinality.
\end{enumerate}	
\end{Proposition}

\begin{proof}
(1) The integral closure of
an integrally closed domain in a finite extension of its fraction field has finite fibers \cite[Chap.\ V, $\mathsection 2.3$, Cor.\ 2]{Bou98}. Since $\Mspec(A)$ is 
the closed fiber of $R \subset A$, it is finite.

\noindent (2) $\Mspec(A\otimes_R R^h)$
is in bijection with the finite set $\Mspec(A)$ by Lemma \ref{lem:Henselization-base-change-integral}.

\noindent (3) 
$\mathfrak M$ contracts to a maximal ideal $\mathfrak P$ in $A$ 
 by Lemma 
\ref{lem:Henselization-base-change-integral}, and so, $\mathfrak P$ 
contracts to $\mathfrak m$ in $R$.  Since
$A_{\mathfrak P}$ is integrally closed, $(A_{\mathfrak P})^h$, is
also an integrally closed domain 
\cite[\href{https://stacks.math.columbia.edu/tag/06DI}{Tag 06DI}]{stacks-project}.
By Lemma \ref{lem:Henselization-base-change}, 
\[
(A_{\mathfrak{P}})^h \cong ((A \otimes_R R^h)_{\mathfrak M})^h.
\]
Since 
$(A \otimes_R R^h)_{\mathfrak M} \rightarrow ((A \otimes_R R^h)_{\mathfrak M})^h$
is faithfully flat,  
descent of integral closedness \cite[\href{https://stacks.math.columbia.edu/tag/033G}{Tag 033G}]{stacks-project}
implies $(A \otimes_R R^h)_{\mathfrak M}$ is an integrally closed domain.

\noindent (4) This follows from part (4) of Lemma \ref{lem:minimal-primes-base-change-Henselization} because $L$ is the 
fraction field of $A$.

\noindent (5) 
A maximal ideal $\mathfrak{M} \in A \otimes_R R^h$ contains a unique 
minimal prime because $(A \otimes_R R^h)_{\mathfrak M}$ is a domain
by part (3).
Let $\mathfrak{p}_1, \dots, \mathfrak{p}_k$ be the minimal primes of
$A \otimes_R R^h$. Then part (3) and Lemma \ref{lem:direct-product-domains} imply that
\begin{equation}
\label{eq:finite-product-Henselian}
A \otimes_R R^h \cong (A \otimes_R R^h)/\mathfrak{p}_1 \times \dots \times (A \otimes_R R^h)/\mathfrak{p}_k.
\end{equation}
Lemma \ref{lem:minimal-primes-base-change-Henselization} 
shows that a minimal prime of $A \otimes_R R^h$ contracts to $(0)$ in $R^h$. 
Thus, the composition
\[
R^h \xrightarrow{\varphi \otimes \id_{R^h}} A \otimes_R R^h \twoheadrightarrow (A \otimes_R R^h)/\mathfrak{p}_i
\]
is an integral extension for all $i = 1, \dots, k$, and so, 
each $(A \otimes_R R^h)/\mathfrak{p}_i$
is a Henselian local domain by Lemma \ref{lem:Nagata-hensel-domains}.
Hence each minimal
prime $\mathfrak{p}_i$ is contained in a unique maximal ideal, say $\mathfrak{M}_i$.
Since $(A \otimes_R R^h)_{\mathfrak{M}_i}$ is a domain by part (3), 
the kernel of
$A \otimes_R R^h \rightarrow  (A \otimes_R R^h)_{\mathfrak{M}_i}$ has to be
$\mathfrak{p}_i$. Moreover, $(A \otimes_R R^h)/\mathfrak{p}_i$ is local with
maximal ideal $\mathfrak{M}_i/\mathfrak{p}_i$, so the induced injection
$(A \otimes_R R^h)/\mathfrak{p}_i \hookrightarrow (A \otimes_R R^h)_{\mathfrak{M}_i}$
is an isomorphism. In particular,
$(A \otimes_R R^h)_{\mathfrak{M}_i}$ 
is a Henselian local domain. 

\noindent (6) The uniqueness and maximality of $\mathfrak{M}_i$
follow from Lemma \ref{lem:Henselization-base-change-integral}, as does the fact that 
$\Mspec(A \otimes_R R^h) = \{\mathfrak{M}_1, \dots, \mathfrak{M}_n\}$. The decomposition $(\ref{eq:finite-product-Henselian})$
and part (5) then show that
\[
A \otimes_R R^h \cong (A \otimes_R R^h)_{\mathfrak{M}_1} \times \dots \times (A \otimes_R R^h)_{\mathfrak{M}_n},
\]
and that each $(A \otimes_R R^h)_{\mathfrak{M}_i}$ is Henselian. Thus, 
$((A \otimes_R R^h)_{\mathfrak{M}_i})^h \cong (A \otimes_R R^h)_{\mathfrak{M}_i}$.
On the other hand, Lemma \ref{lem:Henselization-base-change} implies that
$((A \otimes_R R^h)_{\mathfrak{M}_i})^h \cong (A_{\mathfrak{m}_i})^h$. Thus,
$A \otimes_R R^h \cong (A_{\mathfrak{m}_1})^h \times \dots \times (A_{\mathfrak{m}_n})^h.$

\noindent(7) All the sets are finite because of
parts (1), (4) and the bijections of Lemma \ref{lem:Henselization-base-change-integral} and
Lemma \ref{lem:minimal-primes-base-change-Henselization}. It remains to check that $|\Mspec(A \otimes_R R^h)|  = |\{$minimal prime of $A\otimes_R R^h\}|$. This follows by part (5).
\end{proof}

\begin{Corollary}
\label{cor:colimit}
Let $(R, \mathfrak m)$ be a local domain that is integrally closed in its 
fraction field $K$.
Let $L$ be a finite field extension of $K$ and $A$ be the integral closure of 
$R$ in $L$. Suppose $\Mspec(A) = \{\mathfrak{m}_1,\dots,\mathfrak{m}_n\}$ (the maximal
ideals are assumed to be distinct).
\begin{enumerate}
	\item Let $\Sigma$ be the collection of finite (equivalently, finitely generated) 
	$R$-subalgebras $B$ of $A$ such that $\Frac(B) = \Frac(A) = L$ and 
	$\mathfrak{m}_i \cap B \neq \mathfrak{m}_j \cap B$, for $i \neq j$. Then $\Sigma$ is filtered
	under inclusion and 
	$$A = \colim_{B \in \Sigma} B.$$
	
	\item Let $\mathfrak{M} \in \Mspec(A \otimes_R R^h)$ and $B \in \Sigma$. If $\mathfrak{M}_B$
	is the contraction of $\mathfrak{M}$ to the subring $B \otimes_R R^h$ of $A \otimes_R R^h$,
	then the induced map on local rings 
	$$(B \otimes_R R^h)_{\mathfrak{M}_B} \rightarrow (A \otimes_R R^h)_{\mathfrak M}$$
	is injective, $(B \otimes_R R^h)_{\mathfrak{M}_B}$ is a Henselian domain, and
	$$(A \otimes_R R^h)_{\mathfrak M} = \colim_{B \in \Sigma} (B \otimes_R R^h)_{\mathfrak{M}_B}.$$
\end{enumerate}
\end{Corollary}

\begin{proof}
(1) Since every element of $\Sigma$ is integral over $R$, finitely generated is equivalent to being module finite as an $R$-algebra.
Note that $\Sigma$ is non-empty. Indeed, since $\Frac(A) = L$ is a finite extension of $K$, one can choose a $K$-basis of $L$ consisting of elements 
$b_1,\dots,b_m \in A$. By prime avoidance, for all $i = 1, \dots, n$, choose
$a_i \in \mathfrak{m}_i$ such that $a_i$ is not contained in any of the other maximal ideals
of $A$ (here we need that $A$ is semi-local).
Then by construction, the $R$-subalgebra $R[a_1,\dots,a_n,b_1,\dots,b_m]$ of $A$
is an element of $\Sigma$. 

If $B \in \Sigma$, then any finitely generated
$B$-subalgebra $C$ of $A$ is also in $\Sigma$. Therefore if
$B_1, B_2 \in \Sigma$, then so is $B_1[B_2]$, that is, $\Sigma$ is filtered under
inclusion. Since $A$ is the filtered union of finitely generated 
$B$-subalgebras for any $B \in \Sigma$ and $\Sigma \neq \emptyset$, 
we have $A = \colim_{B \in \Sigma} B$.

\noindent(2) Fix $B \in \Sigma$. Since $B \hookrightarrow A$ is an integral extension,
each $\mathfrak{m}_i \cap B$ is a maximal ideal of $B$. Furthermore, since
every maximal ideal of $B$ is contracted from a maximal ideal of $A$, we get
\[
\Mspec(B) \coloneqq \{\mathfrak{m}_1 \cap B, \dots, \mathfrak{m}_n \cap B\}.
\]
Let $\mathfrak{M}_i \in \Spec(A \otimes_R R^h)$ be the unique prime ideal that
contracts to $\mathfrak{m}_i$. Then 
Lemma \ref{lem:Henselization-base-change-integral} shows 
that
\[
\Mspec(A\otimes_R R^h) = \{\mathfrak{M}_1,\dots,\mathfrak{M}_n\}.
\]
If $(\mathfrak{M}_i)_B$ is the contraction of $\mathfrak{M}_i$ to $B \otimes_R R^h$,
then using the integrality of the extension
$B \otimes_R R^h \hookrightarrow A \otimes_R R^h$ we conclude that
\[
\Mspec(B \otimes_R R^h) = \{(\mathfrak{M}_1)_B,\dots,(\mathfrak{M}_n)_B\}.
\]
The defining property of $\Sigma$ implies that for $i \neq j$, 
$\mathfrak{m}_i \cap B \neq \mathfrak{m}_j \cap B$. As $(\mathfrak{M}_i)_B$ lies
over $\mathfrak{m}_i \cap B$ by the commutativity of the diagram
\[
\begin{tikzcd}
A \arrow[r, ""] 
& A \otimes_R R^h  \\ 
B \arrow[r, ""] \arrow[u, hook]
&B \otimes_R R^h \arrow[u, hook],
\end{tikzcd}
\]
we have $(\mathfrak{M_i})_B \neq (\mathfrak{M_j})_B$ for $i \neq j$. In other
words, $B \otimes_R R^h$ consists of $n$ distinct maximal ideals $(\mathfrak{M}_i)_B$
for $i = 1, \dots, n$, and $\mathfrak{M}_i$ is the unique prime ideal of
$A \otimes_R R^h$ that contracts to $(\mathfrak{M}_i)_B$.

Since $B$ is a finite extension of $R$, $B \otimes_R R^h$ is a finite extension of $R^h$.
The decomposition of finite extensions of
Henselian local domains 
\cite[\href{https://stacks.math.columbia.edu/tag/04GH}{Tag 04GH}]{stacks-project} gives us 
that
\begin{equation}
\label{eq:decomp}
B \otimes_R R^h \cong (B \otimes_R R^h)_{(\mathfrak{M}_1)_B} \times \dots \times 
(B \otimes_R R^h)_{(\mathfrak{M}_n)_B},
\end{equation}
and that each 
$(B \otimes_R R^h)_{(\mathfrak{M}_i)_B}$ is a Henselian local ring. Moreover, $B \otimes_R R^h$
is a subring of $A \otimes_R R^h$, which 
is reduced because it decomposes as a finite product of
domains by part (6) of Proposition \ref{prop:decomposition-henselization-base-change}. 
Thus, $(B \otimes_R R^h)_{(\mathfrak{M}_i)_B}$ is reduced for all $i$.

Applying part (3) of Lemma 
\ref{lem:minimal-primes-base-change-Henselization} to the integral
extension $R \hookrightarrow B$ we see that the number of minimal primes
of $B \otimes_R R^h$ equals the cardinality of $\Spec(L \otimes_K \Frac(R^h))$.
But
\[
|\Spec(L \otimes_K \Frac(R^h))| = |\Mspec(A)| = n
\]
by part (7) of Proposition
\ref{prop:decomposition-henselization-base-change} and 
the fact that $\Frac(A) = \Frac(B) = L$. Consequently,
each factor in the decomposition (\ref{eq:decomp}) has
exactly one minimal prime. Combined with reducedness, it 
follows that each 
$(B \otimes_R R^h)_{(\mathfrak{M}_i)_B}$ 
is a domain. 


In particular, $(\mathfrak{M}_i)_B$ contains a unique minimal prime (which expands to
the zero ideal in $(B \otimes_R R^h)_{(\mathfrak{M}_i)_B}$). Using part (5) of Proposition \ref{prop:decomposition-henselization-base-change},
if $\mathfrak{P}_i$ is the unique minimal prime of $A \otimes_R R^h$ 
contained in $\mathfrak{M}_i$, then 
$$(\mathfrak{P}_i)_B \coloneqq \mathfrak{P}_i \cap (B \otimes_R R^h)$$ 
must be the unique
minimal prime of $B \otimes_R R^h$ contained in $(\mathfrak{M}_i)_B$. 
Indeed, by part (2) of Lemma \ref{lem:minimal-primes-base-change-Henselization}, $\mathfrak{P}_i$ contracts to $(0)$ in $R^h$. Thus $(\mathfrak{P}_i)_B$ also
contracts to $(0)$ in $R^h$. Applying part (2) of Lemma \ref{lem:minimal-primes-base-change-Henselization} again, but
this time to the integral extension $R \hookrightarrow B$, 
then shows that $(\mathfrak{P}_i)_B$ is a minimal prime of $B \otimes_R R^h$.

Using the commutative diagram
\[
\begin{tikzcd}
B \otimes_R R^h \arrow[r, hook] \arrow[d]
& A \otimes_R R^h \arrow[d] \\ (B \otimes_R R^h)_{(\mathfrak{M}_i)_B} \arrow[r, ""]
&(A \otimes_R R^h)_{\mathfrak{M}_i},
\end{tikzcd}
\]
it follows that $\mathfrak{P}_i(A \otimes_R R^h)_{\mathfrak{M}_i}$ must contract to  
$(\mathfrak{P}_i)_B(B \otimes_R R^h)_{(\mathfrak{M}_i)_B}$ in $(B \otimes_R R^h)_{(\mathfrak{M}_i)_B}$. 
Since $(A \otimes_R R^h)_{\mathfrak{M}_i}$
and $(B \otimes_R R^h)_{(\mathfrak{M}_i)_B}$ are both domains (the former ring is a domain
by Proposition
\ref{prop:decomposition-henselization-base-change}), we have 
$\mathfrak{P}_i(A \otimes_R R^h)_{\mathfrak{M}_i} = (0)$ and $(\mathfrak{P}_i)_B(B \otimes_R R^h)_{(\mathfrak{M}_i)_B} = (0)$. 
Thus, the bottom horizontal arrow
is injective.

A maximal ideal $\mathfrak{M}$ of $A \otimes_R R^h$ coincides
with some $\mathfrak{M}_i$. Therefore the argument above shows that for any $B \in \Sigma$,
$(B \otimes_R R^h)_{\mathfrak{M}_B}$
is a Henselian local domain and 
the induced local map 
$$(B \otimes_R R^h)_{\mathfrak{M}_B} \rightarrow (A \otimes_R R^h)_{\mathfrak M}$$
is injective. As tensor product commutes with filtered colimits, we have 
$A \otimes_R R^h = \colim_{B \in \Sigma}
B \otimes_R R^h$
by part (1), and so, 
$(A \otimes_R R^h)_{\mathfrak M} = \colim_{B \in \Sigma} (B \otimes_R R^h)_{\mathfrak{M}_B}$. In this case the filtered colimit is actually
a filtered union because $(B \otimes_R R^h)_{\mathfrak{M}_B}$ is a
subring of $(A \otimes_R R^h)_{\mathfrak M}$, for all $B \in \Sigma$.
\end{proof}

\section{Henselian valuation rings and some ramification theory}
\label{sec:valuation-background}

This section discusses some
background from the ramification theory of extensions of valuations
relevant to Conjecture \ref{conj:Knaf}. We first recall how extensions of valuation rings arise
in algebraic field extensions.

\begin{Proposition}
\label{prop:extensions-valuations}
Let $L/K$ be an algebraic extension of fields. Let $V$ be a 
valuation ring of $K$ and $A$ be the integral closure of $V$ in 
$L$. Then localization $\mathfrak m \mapsto A_{\mathfrak{m}}$ induces a
bijection 
\[
\Mspec(A) \longleftrightarrow \textrm{$\{$valuation rings of $L$ that
dominate $V \}$}.
\]
In particular, if $L/K$ is finite, then there are finitely many
valuation rings of $L$ that dominate $V$.
\end{Proposition}

\begin{proof}[Indication of proof]
For the bijection see \cite[Chap.\ VI, $\mathsection 8.6$,
Prop.\ 6]{Bou98}. If $L/K$ is finite, then 
$A$ has finitely many maximal ideals by part (1)
of Proposition 
\ref{prop:decomposition-henselization-base-change}. 
That there are finitely many valuation rings of $L$ that
dominate $V$ now follows from the bijection of this Proposition.
\end{proof}

As a consequence of the Henselian property, one can now deduce:

\begin{Corollary}
\label{cor:Henselian-valuations}
Let $V$ be a valuation ring of a field $K$. The following are
equivalent:
\begin{enumerate}
	\item $V$ is Henselian.
	\item If $L$ is an algebraic extension of $K$ and $A$ is the
	integral closure of $V$ in $L$, then $A$ is
	the unique 
	valuation ring of $L$ that dominates $V$.
\end{enumerate}
\end{Corollary}

\begin{proof}
(1) $\Rightarrow$ (2) By Lemma \ref{lem:Nagata-hensel-domains}, $A$ must
be a local ring.
 By the bijection of Proposition 
\ref{prop:extensions-valuations} it follows that $A$ must be the
unique valuation ring of $L$ that dominates $V$. 

Conversely, assume (2). Let $B$ be a domain that is an
integral extension of $V$. By Lemma \ref{lem:Nagata-hensel-domains} again, it suffices to show that $B$ is local. Let 
$L = \Frac(B)$. Then $L/K$ is algebraic. If $A$ is the integral 
closure of $V$ in $L$, then $B \subset A$ is integral. Since $A$ is local by (2), $B$ must also be local. Indeed, the unique 
maximal ideal of $A$ must contract to the unique maximal ideal of
$B$ because $B \subset A$ is an integral extension 
\cite[Cor.\ 5.8, Thm.\ 5.10]{AM69}.
\end{proof}

\begin{Remark}
In terms of valuations, Corollary \ref{cor:Henselian-valuations}
can be reinterpreted as saying that if $\nu$ is a valuation of 
a field $K$, then the valuation ring of $\nu$ is Henselian
if and only if for every algebraic extension $L/K$, there exists
a unique valuation 
$\omega$ of $L$ (up to equivalence) that extends $\nu$.
\end{Remark}

The Henselization of a valuation
ring admits a purely valuation theoretic description. However, 
for the purposes of this paper,
it is more helpful to think of Henselizations 
as filtered colimits of local \'etale extensions that induce
isomorphisms on residue fields. One then has the following result.

\begin{Lemma}
\label{lem:valuation-Henselization}
Let $\nu$ be a valuation of a field $K$ with valuation ring
$\mathcal{O}_\nu$ and value group $\Gamma_\nu$. Then the 
Henselization $\mathcal{O}_\nu^h$ of $\mathcal{O}_\nu$ is 
a valuation ring whose associated valuation $\nu^h$ also has
value group $\Gamma_\nu$.
\end{Lemma}

\begin{proof}
See \cite[\href{https://stacks.math.columbia.edu/tag/0ASK}{Tag 0ASK}]{stacks-project}. The main points are that local \'etale
extensions of valuation rings are valuation rings and a filtered
colimit of valuation rings is a valuation ring.
\end{proof}

\begin{Notation}
The fraction field of $\mathcal{O}^h_\nu$ will be denoted by $K^h$.
Thus, $\nu^h$ is a valuation of $K^h$ whose valuation ring is
$\mathcal{O}^h_\nu$.
\end{Notation}

We record a descent result that we will need in the proof of
Theorem \ref{thm:main-theorem}. 


\begin{Lemma}
\label{lem:descent}
Let $\varphi: V \rightarrow W$ be a ring map and
$W$ be a valuation ring. The following are equivalent:
\begin{enumerate}
	\item $V$ is a valuation ring and $\varphi$ is an injective 
	local map.
	\item $\varphi$ is faithfully flat.
	\item $\varphi$ is cyclically pure, that is, for all ideals
	$I$ of $V$, the induced map $V/I \rightarrow W/IW$ is 
	injective.
\end{enumerate}
\end{Lemma}

\begin{proof}
Assume (1). If $\varphi$ is injective, then $W$, being a domain, is a 
torsion-free $V$-module, hence flat \cite[Chap.\ VI, $\mathsection 3.6$, Lem.\ 1]{Bou98}. Since $\varphi$ is local, $\varphi$
is faithfully flat. Thus, (1) $\Rightarrow$ (2). Furthermore,
(2) $\Rightarrow$ (3) is a property of faithfully flat
maps; see \cite[Chap. I, $\mathsection 3.5$, Prop. 9]{Bou98}.

Assume (3). Taking $I = (0)$, we see that $\varphi$ is 
injective. Thus, $V$ is a domain because $W$ is. To show that
$V$ is a valuation ring, it is enough to show that for all 
$x, y \in V$, $xV \subseteq yV$ or $yV \subseteq xV$. Since $W$ is a valuation ring,
we must have $xW \subseteq yW$ or $yW \subseteq xW$. Cyclic purity of $\varphi$ implies that $\varphi^{-1}(IW) = I$,
for any ideal $I$ of $V$. Thus, if $xW \subseteq yW$, then 
$xV = \varphi^{-1}(xW) \subseteq \varphi^{-1}(yW) = yV$. Similarly, 
$yV \subseteq xV$ if $yW \subseteq xW$. Finally, $\varphi$ is 
local because if $\mathfrak{m}_V$ is the maximal ideal of the valuation ring
$V$, then injectivity of $V/\mathfrak{m}_V \rightarrow W/\mathfrak{m}_VW$ shows $\mathfrak{m}_VW \neq W$.
\end{proof}

Conjecture \ref{conj:Knaf} relates essential finite generation
of extensions of valuation rings to fundamental invariants from
the ramification theory of extensions of valuations. We now 
briefly introduce these invariants. 
Let $L/K$ be a field extension, 
$\omega$ be a valuation of $L$ with value group $\Gamma_\omega$ and $\nu$ be its restriction
to $K$ with value group $\Gamma_\nu$. 
Inclusion induces a local homomorphism of the corresponding
valuation rings
\[
(\mathcal{O}_\nu, \mathfrak{m}_\nu, \kappa_\nu) \hookrightarrow (\mathcal{O}_\omega, \mathfrak{m}_\omega, \kappa_\omega).
\]
Note that $\Gamma_\nu$ is a subgroup of $\Gamma_\omega$ and $\kappa_\nu$ is a subfield
of $\kappa_\omega$. 

Fix a finite field extenion $L/K$. One has the fundamental 
inequality \cite[Chap.\ VI, $\mathsection 8.1$, Lem.\ 2]{Bou98}
\begin{equation}
\label{eq:ram-rel}
[\Gamma_\omega:\Gamma_\nu][\kappa_\omega:\kappa_\nu] \leq [L:K].
\end{equation}
In particular, $[\Gamma_\omega:\Gamma_\nu]$ and 
$[\kappa_\omega:\kappa_\nu]$ are finite invariants of 
$\omega/\nu$. This leads to the following definition.
\begin{Definition}
\label{def:ram-res}
Suppose $L/K$ is a finite extension and consider the extension of valuations $\omega/\nu$.
\begin{itemize}
	\item[(a)] The \emph{ramification index} of $\omega/\nu$, denoted $e(\omega|\nu)$,
is $[\Gamma_\omega : \Gamma_\nu]$.
	\item[(b)] The \emph{inertia index} of $\omega/\nu$, denoted $f(\omega|\nu)$,
is $[\kappa_\omega:\kappa_\nu]$.
	\item[(c)] The \emph{initial index} of $\omega/\nu$, denoted $\epsilon(\omega|\nu)$,
	is the cardinality of the set $\{x \in \Gamma_{\omega,\geq0}: x < \Gamma_{\nu, > 0}\}$.
\end{itemize}
\end{Definition}

The finiteness of the initial index follows from the 
inequality
\[
\epsilon(\omega|\nu) \leq e(\omega|\nu),
\]
which holds because if $x, y \in \Gamma_{\omega, \geq 0}$ are 
distinct elements such that $x, y < \Gamma_{\nu, > 0}$, then
$x + \Gamma_\nu \neq y + \Gamma_\nu$ in $\Gamma_\omega/\Gamma_\nu$. 
Indeed, assume without loss of generality that
$0 \leq x < y$. Then $y - x \in \Gamma_{\omega, > 0}$ and
$y - x \leq y < \Gamma_{\nu, > 0}$, that is, $y - x \notin 
\Gamma_\nu$.  

By Lemma \ref{lem:valuation-Henselization},
if $\omega/\nu$ is an extension of valuations, then for the 
extension of Henselizations $\omega^h/\nu^h$, we have 
$e(\omega|\nu) = e(\omega^h|\nu^h), 
f(\omega|\nu) = f(\omega^h|\nu^h)$ and $\epsilon(\omega|\nu) = 
\epsilon(\omega^h|\nu^h)$ because Henselizations 
do not alter value groups and residue fields. 
In addition, one can use the isomorphism of part (6) of 
Proposition \ref{prop:decomposition-henselization-base-change} 
and Proposition \ref{prop:extensions-valuations} 
to conclude that $L \otimes_K K^h$ is a finite product of fields,
one of which coincides with $L^h$, the fraction field of 
$\mathcal{O}^h_\omega$. Thus, 
\begin{equation}
\label{eq:henselization-degree}
[L^h:K^h] \leq [L:K] < \infty.
\end{equation}
Using these observations we recall the notion of the defect of $\omega/\nu$ \cite{Kuh11},
which measures to what extent equality fails in (\ref{eq:ram-rel}), at least when $\omega$ is the unique extension of $\nu$ to $L$.


\begin{Definition}
\label{def:defect}
Let $L/K$ be a finite field extension and $\nu$ be a valuation of $K$. If $\omega$ is the unique extension of
$\nu$ to $L$, then the \emph{defect of $\omega/\nu$},
denoted $d(\omega|\nu)$, is defined to be
\[
d(\omega|\nu) = \frac{[L:K]}{e(\omega|\nu)f(\omega|\nu)}.
\]
If the extension of valuations $\omega/\nu$ is not necessarily
unique, the \emph{defect of $\omega/\nu$} is defined to be the 
defect of the extension of Henselizations $\omega^h/\nu^h$, that 
is,
$$d(\omega|\nu) = \frac{[L^h:K^h]}{e(\omega|\nu)f(\omega|\nu)}.$$
We say $\omega/\nu$ is \emph{defectless} if $d(\omega|\nu) = 1$,
that is, if $[L^h:K^h] = e(\omega|\nu)f(\omega|\nu)$.
\end{Definition}

\begin{Remark} 
{\*}
\begin{enumerate}
	\item If $L/K$ is a finite extension, then $\omega^h$ is 
	the unique extension of $\nu^h$ to $L^h$ by Corollary 
	\ref{cor:Henselian-valuations} and (\ref{eq:henselization-degree}). Thus, the definition of the defect of an extension of
	valuations that is not necessarily unique in terms of 
	the defect of the extension of henselizations makes sense.


	\item If $\omega$ is the unique extension $\nu$ to $L$, then
	$d(\omega|\nu) = d(\omega^h|\nu^h)$. Thus, the two 
	notions of defect are consistent for unique
	extensions of valuations. 
	The only thing we 
	need to check is that $[L^h:K^h] = [L:K]$. By Proposition
	\ref{prop:extensions-valuations} and uniqueness of 
	the extension $\omega/\nu$, 
	the integral closure of 
	$\mathcal{O}_\nu$ in $L$ must be $\mathcal{O}_\omega$. Then by 
	part (6) of Proposition \ref{prop:decomposition-henselization-base-change} applied to $R = \mathcal{O}_\nu$ and 
	$A = \mathcal{O}_\omega$, we get
	$$ 
	\mathcal{O}_\omega \otimes_{\mathcal{O}_\nu} \mathcal{O}_\nu^h
	\cong \mathcal{O}^h_\omega.	
	$$
	Consequently,
	$L \otimes_K K^h \cong L^h$, and so, $[L^h:K^h] = [L:K]$. 
	
	\item If $\kappa_\nu$ has characteristic $0$, 
	then we always have $d(\omega|\nu) = 1$, and if $\kappa_\nu$ has characteristic $p > 0$,
	then $d(\omega|\nu) = p^n$, for some integer $n \geq 0$ \cite[Pg. 280--281]{Kuh11}. 
	Thus, the 
	notion of defect is only interesting in residue characteristic 
	$p > 0$, that is, when $\mathcal{O}_\nu$ has prime or mixed characteristics. 
	Furthermore, $d(\omega|\nu)$ is always a positive integer.

	\item Rephrased in terms of defect, Conjecture \ref{conj:Knaf} 	asserts
	that $\mathcal{O}_\omega$ is essentially of finite type over 
	$\mathcal{O}_\nu$ if and only if $\omega/\nu$ is defectless and
	$e(\omega|\nu) = \epsilon(\omega|\nu)$.
\end{enumerate}
\end{Remark}

\begin{Example}
Let $K$ be a field of characteristic $p > 0$ for which $[K:K^p]< \infty$ and $\nu$ be a 
valuation of $K$. If $\nu^p$ denotes the restriction of $\nu$
to the subfield $K^p$ of $K$, then using pure inseparability of
the extension $K/K^p$ one can verify that $\nu$ is the unique
extension of $\nu^p$ to $K$. Then
$$d(\nu|\nu^p) = \frac{[K:K^p]}
{[\Gamma_\nu:p\Gamma_\nu][\kappa_\nu:\kappa^p_\nu]}.$$
The defect of $\nu/\nu^p$ controls interesting properties of
$\nu$. For example, it is shown in \cite[Proof of 
Thm.\ 5.1]{DS16} (see also \cite[Cor.\ IV.23]{Dat18}) that if 
$K$ is a function field of a variety over a ground field $k$, 
and $\nu$ is a valuation of $K/k$, 
then $\nu/\nu^p$ is defectless if and only if $\nu$ is an Abhyankar valuation of $K/k$.
\end{Example}

\section{Proof of Theorem \ref{thm:main-theorem}}
\label{sec:proof}

We recall the statement of Theorem \ref{thm:main-theorem} for the reader's 
convenience.

\begin{customthm}{\ref{thm:main-theorem}}
Let $L/K$ be a finite field extension, $\omega$ be a valuation of 
$L$ with valuation ring $(\mathcal{O}_\omega, \mathfrak{m}_\omega, \kappa_\omega)$ 
and value group $\Gamma_\omega$, and $\nu$ be the restriction of $\omega$ to $K$
with valuation ring $(\mathcal{O}_\nu, \mathfrak{m}_\nu, \kappa_\nu)$
and value group $\Gamma_\nu$. 
Let $\mathcal{O}^h_\omega$ (resp. $\mathcal{O}^h_\nu$) be the Henselization of $\mathcal{O}_\omega$ (resp. $\mathcal{O}_\nu$)
with fraction field $L^h$ (resp. $K^h$). 
Then the following are equivalent.
\begin{enumerate}
	\item[(i)] $\mathcal{O}_\omega$ is essentially of finite
	presentation over $\mathcal{O}_\nu$.
	\item[(ii)] $\mathcal{O}_\omega$ is essentially of finite type
	over $\mathcal{O}_\nu$.
	\item[(iii)] $[L^h:K^h] = [\Gamma_\omega:\Gamma_\nu][\kappa_\omega:\kappa_\nu]$ and $\epsilon(\omega/\nu) = [\Gamma_\omega:\Gamma_\nu]$.
	\item[(iv)] $\mathcal{O}^h_\omega$ is essentially of finite
	type over $\mathcal{O}^h_\nu$.
	\item[(v)] $\mathcal{O}^h_\omega$ is a module finite 
	$\mathcal{O}^h_\nu$-algebra.
\end{enumerate}
\end{customthm}




\begin{proof}
(i) $\Rightarrow$ (ii) is clear, and (ii) $\Rightarrow$ (iii),
(iv) $\Rightarrow$ (iii) 
were established by Knaf \cite[Thm.\ 4.1]{CN19} (see also 
Remark \ref{rem:ZMT}). 
Since 
$\omega^h$ is the unique extension of $\nu^h$ to $L^h$ by 
Corollary \ref{cor:Henselian-valuations}, the implication (iii) $\Rightarrow$ (v)
 follows by \cite[Cor.\ 2.2]{CN19} because $\mathcal{O}^h_\omega$ is the
integral closure of $\mathcal{O}^h_\nu$ in $L^h$. Furthermore, clearly (v) $\Rightarrow$
(iv). 
 
It remains to show that (v) $\Rightarrow$ (i). Let $A$ be the 
integral closure of $\mathcal{O}_\nu$ in $L$. By 
Proposition \ref{prop:extensions-valuations}, let $\mathfrak{m}$
be the unique maximal ideal of $A$ such that
\[
\mathcal{O}_\omega = A_{\mathfrak{m}}.
\]
Let $\Sigma$ be the collection of finite $\mathcal{O}_\nu$-subalgebras $B$
of $A$ as in Corollary \ref{cor:colimit}, and we let 
\[
\textrm{$\mathfrak{m}_B \coloneqq$ the maximal ideal $\mathfrak{m} \cap B$ of $B$.}
\] 
Then for all
$B \in \Sigma$,
\[
\textrm{$A_{\mathfrak m}$ dominates 
$B_{\mathfrak{m}_B}$ and $B_{\mathfrak{m}_B}$ dominates $\mathcal{O}_\nu$.}
\]
By Lemma
\ref{lem:Henselization-base-change-integral}, let 
\[
\textrm{$\mathfrak{M} \coloneqq$
the unique prime (equivalently, maximal) ideal of 
$A \otimes_{\mathcal{O}_\nu} \mathcal{O}^h_\nu$ that contracts
to $\mathfrak{m}$}.
\] 
For all $B \in \Sigma$, let $\mathfrak{M}_B$ denote
the contraction of $\mathfrak{M}$ to the 
$\mathcal{O}^h_\nu$-subalgebra 
$B \otimes_{\mathcal{O}_\nu} \mathcal{O}^h_\nu$ of 
$A \otimes_{\mathcal{O}_\nu} \mathcal{O}^h_\nu$ (it is
a subalgebra by flatness of $\mathcal{O}^h_\nu$). Then
by the commutativity of the diagram
\[
\begin{tikzcd}
A \arrow[r, ""] 
& A \otimes_{\mathcal{O}_\nu} {\mathcal{O}^h_\nu}  \\ B \arrow[r, ""]\arrow[u, hook]
&B \otimes_{\mathcal{O}_\nu} \mathcal{O}^h_\nu \arrow[u, hook],
\end{tikzcd}
\]
and Lemma \ref{lem:Henselization-base-change-integral} again,
 $\mathfrak{M}_B$ is the unique
prime (equivalently, maximal) ideal of $B \otimes_{\mathcal{O}_\nu} \mathcal{O}^h_\nu$
that contracts to the maximal ideal $\mathfrak{m}_B$ of $B$.

By Corollary \ref{cor:colimit}, for all 
$B \in \Sigma$, 
\[
(B \otimes_{\mathcal{O}_\nu} \mathcal{O}^h_\nu)_{\mathfrak{M}_B}
\]
is a $\mathcal{O}^h_\nu$-subalgebra of 
\[
(A \otimes_{\mathcal{O}_\nu} \mathcal{O}^h_\nu)_{\mathfrak M},
\]
and 
\[
(A \otimes_{\mathcal{O}_\nu} \mathcal{O}^h_\nu)_{\mathfrak M} =
\colim_{B \in \Sigma} (B \otimes_{\mathcal{O}_\nu} \mathcal{O}^h_\nu)_{\mathfrak{M}_B}.
\]
Note that the filtered colimit is a filtered union. By part (6) of Proposition \ref{prop:decomposition-henselization-base-change}, we have that 
\[
(A \otimes_{\mathcal{O}_\nu} \mathcal{O}^h_\nu)_{\mathfrak M} \cong
(A_{\mathfrak m})^h = \mathcal{O}^h_\omega.
\]
Since $\mathcal{O}^h_\omega$ is a module-finite 
$\mathcal{O}^h_\nu$-algebra by the hypothesis of (v), 
we can find $B \in \Sigma$ such that
\[
(A \otimes_{\mathcal{O}_\nu} \mathcal{O}^h_\nu)_{\mathfrak M} = 
(B \otimes_{\mathcal{O}_\nu} \mathcal{O}^h_\nu)_{\mathfrak{M}_B}.
\]
Therefore $(B \otimes_{\mathcal{O}_\nu} \mathcal{O}^h_\nu)_{\mathfrak{M}_B}$ is a Henselian valuation ring, and  by
part (2) of Lemma \ref{lem:Henselization-base-change}, we conclude
\[
(B_{\mathfrak{m}_B})^h \cong ((B \otimes_{\mathcal{O}_\nu} \mathcal{O}_\nu^h)_{\mathfrak{M}_B})^h = 
(B \otimes_{\mathcal{O}_\nu} \mathcal{O}_\nu^h)_{\mathfrak{M}_B}.
\]
In other words, $(B_{\mathfrak{m}_B})^h$ is a valuation ring, so by
descent (Lemma \ref{lem:descent}), $B_{\mathfrak{m}_B}$ is a 
valuation ring as well. By the definition of the collection 
$\Sigma$, we have 
\[
\Frac(B_{\mathfrak{m}_B}) = \Frac(B) = L,
\]
that is, $B_{\mathfrak{m}_B}$ is a valuation ring of $L$. Since $\mathcal{O}_\omega = A_{\mathfrak{m}}$ is also a valuation ring
of $L$ that dominates $B_{\mathfrak{m}_B}$, we 
must have
\[
\mathcal{O}_\omega = B_{\mathfrak{m}_B}
\]
because valuation rings are maximal with respect to domination of local rings. Thus, 
$\mathcal{O}_\omega$ is the localization of the finite 
$\mathcal{O}_\nu$-algebra $B$. But $B$ is $\mathcal{O}_\nu$-flat
since it is a torsion-free $\mathcal{O}_\nu$-module 
\cite[Chap.\ VI, $\mathsection 3.6$, Lem.\ 1]{Bou98}. Therefore,
$B$ is a finitely presented $\mathcal{O}_\nu$-algebra by 
Lemma \ref{lem:finite-pres}. This completes the proof of (v) 
$\Rightarrow$ (i), hence also of the Theorem.
\end{proof}

The proof of Theorem \ref{thm:main-theorem} establishes
the stronger result that if conditions (1) and (2) of Conjecture
\ref{conj:Knaf} hold, then $\mathcal{O}_\omega$ is the localization
of a finite $\mathcal{O}_\nu$-algebra $B$ contained in
the integral closure of $\mathcal{O}_\nu$ in $L$. Since a finitely
generated torsion-free module over a valuation ring is free, $B$ is 
a free $\mathcal{O}_\nu$-module of finite rank.

\begin{Remark}
\label{rem:ZMT}
One can prove (i) $\Rightarrow$ (v) (or, (ii) $\Rightarrow$ (v))
using Zariski's Main Theorem. 
Suppose $\mathcal{O}_\omega$ 
is the localization of a finite type 
$\mathcal{O}_\nu$-algebra $B$ at a prime ideal $\mathfrak{p}$. Then
\[
[\kappa(\mathfrak{p}):\kappa_\nu] = [\kappa_\omega:\kappa_\nu]
\leq [L:K] < \infty,
\]
where the first inequality follows from (\ref{eq:ram-rel}).
Moreover, 
\[
\dim(B_{\mathfrak p}/\mathfrak{m}_\nu B_{\mathfrak p}) = 
\dim(\mathcal{O}_\omega/\mathfrak{m}_\nu \mathcal{O}_\omega) = 0
\]
because $\mathfrak{m}_\omega$ is the only prime ideal of 
$\mathcal{O}_\omega$ that contracts to $\mathfrak{m}_\nu$ (if not,
a non-maximal prime of $\mathcal{O}_\omega$ that contracts to
$\mathfrak{m}_\nu$ will give a non-maximal prime of the integral 
closure $A$ of $\mathcal{O}_\nu$ in $L$ that contracts to 
the maximal ideal $\mathfrak{m}_\nu$). Thus,
$B$ is quasi-finite at $\mathfrak p$ by part (6) of
\cite[\href{https://stacks.math.columbia.edu/tag/00PK}{Tag 00PK}]{stacks-project}. Then Zariski's Main Theorem 
\cite[\href{https://stacks.math.columbia.edu/tag/00QB}{Tag 00QB}]{stacks-project} implies that
there exists a finite $\mathcal{O}_\nu$-subalgebra $B'$ of $B$
such that $\mathcal{O}_\omega$ is a localization of $B'$ at a 
maximal ideal $\mathfrak q$ of $B'$ ($\mathfrak{q}$ is maximal
because it contracts to $\mathfrak{m}_\nu$). By 
Lemma \ref{lem:Henselization-base-change}, $\mathcal{O}_\omega^h
 = (B')_\mathfrak{q}^h$
is the Henselization of $(B' \otimes_{\mathcal{O}_\nu} \mathcal{O}^h_\nu)_{\mathfrak Q}$,
where $\mathfrak{Q}$ is the unique prime 
ideal of $B' \otimes_{\mathcal{O}_\nu} \mathcal{O}^h_\nu$ that
contracts to $\mathfrak{q}$. But $\mathfrak{q}$ is maximal, so
$\mathfrak{Q}$ is a maximal ideal as well by Lemma \ref{lem:Henselization-base-change-integral}. Since 
$B' \otimes_{\mathcal{O}_\nu} \mathcal{O}_\nu^h$ is a finite
$\mathcal{O}_\nu^h$-algebra, it decomposes as a finite product of
finite Henselian local rings 
\cite[\href{https://stacks.math.columbia.edu/tag/04GG}{Tag 04GG}]{stacks-project}. Then $(B' \otimes_{\mathcal{O}_\nu} \mathcal{O}^h_\nu)_{\mathfrak Q}$ 
must coincide with one these local factors, that is,
$(B' \otimes_{\mathcal{O}_\nu} \mathcal{O}^h_\nu)_{\mathfrak Q}$  is
Henselian and finite. Consequently, 
$\mathcal{O}^h_\omega \cong ((B' \otimes_{\mathcal{O}_\nu} \mathcal{O}^h_\nu)_{\mathfrak Q})^h = (B' \otimes_{\mathcal{O}_\nu} \mathcal{O}^h_\nu)_{\mathfrak Q}$ is a finite
$\mathcal{O}^h_\nu$-algebra, proving (v). 
Once we know that $\mathcal{O}^h_\omega$
is a module-finite $\mathcal{O}^h_\nu$ algebra, part (iii) of 
Theorem \ref{thm:main-theorem} now follows by \cite[Chap.\ VI,
$\mathsection 8.5$, Thm.\ 2]{Bou98}. This gives a different proof
of \cite[Thm.\ 4.1]{CN19}.
\end{Remark}

\bibliographystyle{amsalpha}

\end{document}